\documentclass[11pt]{amsart}
\usepackage{geometry}                % See geometry.pdf to learn the layout options. There are lots.
\usepackage{graphicx}
\usepackage{amssymb}
\usepackage{epstopdf}
\newtheorem{thm}{Theorem}

\newtheorem{lem}{Lemma}
\newtheorem{cor}{Corollary}

\title{2-recognizeable classes of Leibniz Algebras}
\author{Tiffany Burch$^{\MakeLowercase{a}}$, Meredith Harris$^{\MakeLowercase{a}}$, Allison McAlister$^{\MakeLowercase{a}}$, Elyse Rogers$^{\MakeLowercase{b}}$, Ernie Stitzinger$^{\MakeLowercase{a}}$, S. McKay Sullivan$^{\MakeLowercase{a}}$}
\date{}                                           % Activate to display a given date or no date

\begin{document}
\maketitle
\begin{center}
$^a$Department of Mathematics, North Carolina State University,
\\
P.O. Box 8205, Raleigh, North Carolina 27695-8205
\\
\emph{Email address}: tmmyers@ncsu.edu,maharr11@ncsu.edu, armcalis@ncsu.edu, stitz@ncsu.edu, smsulli4@ncsu.edu
\\
$^b$ Surrey University
\\
\emph{Email address}: esrogers@ncsu.edu
\end{center}
%\section{}
%\subsection{}

ABSTRACT

\vspace{5mm}

We show that for fields that are of characteristic 0 or algebraically closed of characteristic greater than 5, that certain classes of Leibniz algebras are 2-recognizeable. These classes are solvable, strongly solvable and supersolvable. These same results hold in Lie algebras and in general for groups. 

\vspace{5mm}


\vspace{5mm}

I. PRELIMINARIES

\vspace{5mm}

A property of algebras is called n-recognizeable if whenever all the n generated subalgebras of algebra $L$ have the property, then $L$ also has the property. An analogous definition holds for classes of groups. In Lie algebras, nilpotency is 2-recognizeable due to Engel's theorem and the same holds for Leibniz algebras. For Lie algebras, solvability, strong solvability and supersolvability are 2-recognizeable when they are taken over a field of characteristic 0 or an algebraically closed field of characteristic greater than 5. These results are shown in [7] and [12] using different methods. The purpose of this work is to extend these results to Leibniz algebras. Corresponding results in group theory are shown in [8] and [9].

\vspace{5mm}

The definition of Leibniz algebra can be given in terms of the left multiplications being derivations. A theme in this work is that assumptions will be given in terms of the left multiplications. Thus, that nilpotency is 2-recognizeable in Leibniz algebras  follows from all left multiplications being nilpotent, Engel's theorem. This result, shown in several places, can be cast as in Jacobson's refinement to Engel's theorem for Lie algebras, see [6], a result that we use.

\vspace{5mm}

 For the supersolvable case, we need the Leibniz algebra version of a Lie algebra result due to Barnes and Newell [2]. This result is a characterization of supersolvability in terms of conditions on the left multiplications and strong solvability. Barnes has extended his Lie algebra result to Leibniz algebras [4], where his conditions are on both left and right multiplications as well as strong solvability. Following the theme in this paper, we will obtain a result where the conditions are only on left multiplications and strong solvability, the latter of which also can be given in terms of left multiplications. Another generalization of Lie algebras, Malcev algebras, also has this type of result [10], [13].

\vspace{5mm}

Let $L$ be a Leibniz algebra and $x \in  L$. We denote left and right multiplication of $L$ by $x$ as $L_x$ and $R_x$. Let $Leib(L)$ be the span of the squares of the elements of $L$. It is an ideal in $L$ [3]. If $S$ is a subset of $L$, then $Z_L^l(S)$ and $Z_L(S)$ will denote the left centralizer and centralizer of $S$ in $L$. If $S$ is an ideal, then both these centralizers are also ideals in $L$.  $\Phi(L)$ will denote the Frattini ideal of $L$. Results on Frattini ideals of Leibniz algebras are found in [5].

\vspace{5mm}

II. SOLVABLE ALGEBRAS

\vspace{5mm}

We first consider solvability for Leibniz algebras. The result is field dependent. For many fields, solvability is 2-recognizeable in Leibniz algebras. In fact, the result in this case is the same as the Lie algebra result [12].

\vspace{5mm}

\begin{thm}\label{thm1}
Solvability is 2-recognizeable for Leibniz algebras over a field of characteristic 0 and over algebraically closed fields of characteristic greater than 5.
\end{thm}

\begin{proof}
Let $L$ be a minimal counterexample. Simple Leibniz algebras are Lie algebras and there are no Lie algebra counterexamples [12].  Thus $L$ is not simple. If $N$ is a proper ideal of $L$, then the hypothesis holds in $N$ and $L/N$. Hence both are solvable and the result holds.
\end{proof}

\vspace{5mm}

III. STRONGLY SOLVABLE  ALGEBRAS

\vspace{5mm}

A Lie algebra whose derived algebra is nilpotent is called strongly solvable. Strong solvability for Lie algebras is 2-recognizeable in the class of solvable Lie algebras, for algebras of characteristic 0 and for algebras over an algebraically closed field of characteristic greater than 5 as has been shown by several methods [7] and [12]. In this section we extend these results to Leibniz algebras. As in Lie theory, we show the result in the class of solvable Leibniz algebras and extend it to the other cases using Theorem \ref{thm1}. The method used in [12] for the Lie cases is to show that the algebra is strongly solvable if and only if $e_n(x,y)=0$ for all $x \mbox{ and }y$ and almost all $n$ where $e_n(x,y)=L^n_{xy}(x)$.  To cope with the lack of symmetry for Leibniz algebras, we also consider $f_n(x,y)=L^n_{xy}(y)$.

\begin{thm}\label{thm2}
Let $L$ be a solvable Leibniz algebra. Then $L$ is strongly solvable if and only if $e_n(x,y)=0$ and $f_n(x,y)=0$ for almost all $n$ and all $x,y \in L$.
\end{thm}

We show two preliminary results first.

\vspace{5mm}

\begin{lem}\label{lem1}
Suppose that $L$ is generated by $x$ and $y$. Suppose that $e_n(x,y)=f_n(x,y)=0$ for almost all $n$. Then $L_{xy}$ is nilpotent acting on $L$.
\end{lem}

\begin{proof}
Let $u \in L$. Then $u$ is a linear combination of elements of the form $z=z_s(\hdots (z_2z_1) \hdots)$  where each $z_j=x$ or $y$. Since $L_{xy}$ is a derivation,
 $$L^t_{xy}(z)= \sum L^{i_s}_{xy}(z_s)(L^{i_{s-1}}_{xy}(z_{s-1}) \hdots (L^{i_{2}}_{xy}(z_{2})L^{i_{1}}_{xy}(z_{1}))\hdots) \mbox{ where } i_s+\hdots+i_1=t.$$ For $t=(n-1)s+1$, at least one of the terms is 0. Since $L$ is finite dimensional, $L$ has a basis made up of terms of the form $z$ and $L_{xy}$ to  a power takes each to 0. Thus $L_{xy}$ is nilpotent on $L$.
\end{proof}

\vspace{5mm}

For $x \in L$, let $L_0(x)$ and $L_1(x)$ be the Fitting 0 and 1 component of $L_x$ acting on $L$. A solvable Leibniz algebra is called primitive if it has a unique minimal ideal $A$ which is its own left centralizer and is complemented. In Lie algebras, all complements of $A$ are conjugate. If $L$ is Leibniz but not Lie, it is shown in [3] that the complement is unique. In our situation, there is a short, natural proof of this result which follows.

\vspace{5mm}

\begin{lem}\label{lem2}
Let $L$ be a solvable, primitive, Leibniz algebra, which is not a Lie algebra, such that $L$ is not strongly solvable but all proper subalgebras and quotients of $L$ are strongly solvable. Let $A$ be the unique minimal ideal of $L$. Then the complement of $A$ in $L$ is unique.
\end{lem}

\begin{proof}
Let $B$ be a complement of $A$ in $L$ and $b \in B \cap L^2$. Under the conditions $L_1(b)\subset A$, hence $L_0(b)$ is a supplement of $A$ in $L$. If $L_0(b) =L$ for all such $b$, then $L^2$ is nilpotent, a contradiction. Hence $B=L_0(b)$ for some $b$. If $C$ is another complement to $A$ in $L$, then $C=L_0(c)$ for some $c \in C \cap L^2$. Then $c=b+a$ for some $a \in A$ and $b \in B \cap L_2$. Since $a \in A=Leib(L), L_b=L_c$. Hence $B=L_0(b)=L_0(c)=C$.
\end{proof}

\vspace{5mm}

\begin{proof}[Proof of Theorem 2]
If $L^{2}$ is nilpotent, then the condition clearly holds. For the converse, suppose that $L$ is a minimal counterexample. Let $A$ be a minimal ideal of $L$. The hypothesis holds in $L/A$, hence $(L/A)^2$ is nilpotent and each $L_{xy}$ to some power takes $L$ to $A$. If $B$ is another minimal ideal of $L$, then $L^2=(L/A\cap B)^2$ is nilpotent, a contradiction. Hence $A$ is the unique minimal ideal in $L$. Furthermore, if $A \in \Phi(L)$ , then $(L/A)^2$ nilpotent, gives that $L^2$ is nilpotent, a contradiction. Hence $\Phi(L)=0$ and $A$ is complemented in $L$ by a subalgebra $B$ [5]. Thus $L$ is primitive. Since $A$ is unique, $A$ is its own left centralizer. We may assume that $L$ is not a Lie algebra for the result is known in that case [12].  Then  $A = Leib(L)$ [5] . Then $B$ is the unique complement of $A$.
\end{proof}

\vspace{5mm}

The set $S = \{xy | x,y \in  L \} \cup L^{(2)}$ is a Lie set whose span is $L^2$ where $L^{(2)}=[L^2,L^2]$. We will show that $L_s$ is nilpotent on $L$ for all $s$ in $S$. Then, using [6], $L^2$ is nilpotent.  Note that $L_s$ is nilpotent for all $s \in L^{(2)}$ since this ideal of $L$ is nilpotent by induction on $L^2$. Let $U$ be the subalgebra generated by $x,y \in L$. If $U=L$, then $L_{xy}$ is nilpotent on $L$ by the lemma. If $U+A \neq L$, then by induction $U^2 \subset (U+A)^2$ is nilpotent. Then a power of $L_{xy}$ takes $L$ to $A$ and a further power takes $L$ to 0. Hence $L_{xy}$ is nilpotent on $L$. 

\vspace{5mm}

Suppose that $L=A+U$. Since we have taken care of the case that $U=L$ and $A \cap U$ is an ideal of $L$, $U$ is a complement of $A$ in $L$. Let $a \in A$, $a \neq  0$ and take $V$ to be the subalgebra generated by $x$ and $y+a$. Since $x, y=(y+a)-a \in V+A$, it follows that $L=A+U=A+V$. Then $V \cap A= 0$ or $A$ since $V \cap A$ is an ideal of $L$ contained in the minimal ideal $A$. If $V \cap A= 0$, then using Lemma \ref{lem2}, $U=V$ since $U$ is the unique complement of $A$ in $L$. Since $y$ and $y+a$ are in $V=U$, $a$ is also, a contradiction. Hence $V=L$. By the lemma, $L_{x(y+a)}$ acts nilpotently on $L$. For each $s$ in the Lie set $T=(x(y+a)) \cup L^{(2)}$ has $L_s$ nilpotent on $L$. Since $xy=x(y+a)-xa$ is in the span of $T$, it follows that $L_{xy}$ is nilpotent on $L$. Now every $s$ in the original Lie set $S$ has $L_s$ acting nilpotently on $L$. Hence $L^2$ acts nilpotently on $L$ and $L^2$ is nilpotent.

\vspace{5mm}

\begin{cor}\label{cor1}
In solvable Leibniz algebras, strong solvability is 2-recognizeable.
\end{cor}

\begin{proof}
Let $x$ and $y$ be in $L$. The subalgebra, $H$,  generated by $x$ and $y$, has  $H^2$ nilpotent. Thus $L_{xy}$ acts nilpotently on $H$. Hence $e_n(x,y)$ and $f_n(x,y)$ are 0 when $n \geq \mbox{dim }H$. Therefore the conditions of the theorem are satisfied and $L^2$ is nilpotent.
\end{proof}

\vspace{5mm}

\vspace{5mm}

This corollary can be extended to larger classes. For if the base field has characteristic 0 or if it is algebraically closed, then solvability of all two generated subalgebras yields solvability of the algebra. Therefore, if all two generated subalgebras have nilpotent derived algebra, then all two generated subalgebra are solvable and $L$ is solvable. Therefore, for the fields considered, if all two generated subalgebras are strongly solvable then the algebra is solvable and the above corollary gives that the algebra is strongly solvable.

\vspace{5mm}

\begin{thm}\label{thm3}
For Leibniz algebras over a field of characteristic 0 or an algebraically closed field of characteristic greater than 5, then strong solvability is 2-recognizeable.
\end{thm}

\vspace{5mm}

IV. SUPERSOLVABLE  ALGEBRAS

\vspace{5mm}

Let $a\in L$. Suppose that $L_a$ has the minimum polynomial $m(x)=\pi_1(x)^{d_1} \dots \pi_k(x)^{d_k}$ where $\pi_i(x)$ is linear for $1 \leq s$ and nonlinear for $i >s$. Then $\pi_i(x)=(x-c_i)^{d_i}$ for $i \leq s$. Let $L^{*}_{c_i}$ for $i \leq s$ and $L^{*}_{\pi_i}$ for $i > s$ be components in the primary decomposition of $L_a$ acting on $L$ and let $L^{*}_0(a)=L^{*}_{c_1} \bigoplus \hdots \bigoplus L^{*}_{c_s} $ and $L^{*}_1(a)=L^{*}_{\pi_{s+1}} \bigoplus \hdots \bigoplus L^{*}_{\pi_t}$ and let $m_1(x)=(x-c_1)\hdots(x-c_s)$. Both $L^{*}_0(a)$ and $L^{*}_1(a)$ are invariant under $L_a$ by the general theory of linear transformations. We also see that $m_1(L_a)$ acts nilpotently on $L^{*}_0(a)$ since $m_1(L_a)^{d_i}=(L_a-c_1I)^{d_1} \hdots (L_a-c_sI)^{d_s}y=0$ for $y \in L^{*}_{c_i}$ since the terms on the right hand side of the equation commute. Hence $m_1(L_a)^{max(d_i)}=0$ on $L^{*}_0(a)$. Also $\mbox{gcd}(m_1(x), \pi_i(x)^{d_i})=1=m_1(x)u(x)+\pi_i(x)v(x)$ for some $u(x), v(x)$. Let $v \in L^{*}_{\pi_i}$  Then $v=Iv=(m_1(L_a)u(L_a)+\pi_i(L_a)^{d_i}v(L_a))v=m_1(L_a)u(L_a)v$. If $m_1(L_a)v=0$ , then $Iv=v=0$. Hence $v=0$. Therefore $m_1(L_a)$ is non-singular on each $L^{*}_{\pi_i}$ for $i>s$ and $m_1(L_a)$ is non-singular on $L^{*}_i(a)$. 
Let $x \in L^{*}_{c_i}$ and $y \in L^{*}_{c_j}$. Then, using the Leibniz identity for derivations, $L_a-((c_i+c_j)I)^{d_i+d_j}(xy)=\sum_{k=0}^{d_i+d_j} {d_i+d_j \choose k} (L_a-c_iI)^kx(L_a-c_jI)^{d_i+d_j-k}y=0$ giving that $xy \in L_{c_i+c_j}$. Therefore $L^{*}_0$ is a subalgebra. We record this result as

\vspace{5mm}

\begin{lem}\label{lem3}
$L^{*}_{0}(a)$ is a subalgebra of $L$ for all $a \in L$.
\end{lem}

\vspace{5mm}

\begin{lem}\label{lem4}
Let $A$ be an ideal of $L$ such that dim$(A)=1$. Then $L^2 \subseteq C_L(A)$.
\end{lem}

\begin{proof}
Let $x,y \in L$ and $a \in A$ with $xa=\alpha_x a$. Since $A$ is a minimal ideal of $L$, then either $ax=-\alpha_xa$ or $ax=0$ for all  $x \in L$. Then $(xy)a=x(ya)-y(xa)=x (\alpha_y a)-y (\alpha_x a)=\alpha_x \alpha_y a - \alpha_y \alpha_x a =0$ or $(xy)a=0.$ Therefore $L^2A=0$. Similarly,
$a(xy)=(ax)y+x(ay)=(-\alpha_x a)y+x (-\alpha_y a)=\alpha_x \alpha_y a - \alpha_y \alpha_x a =0$ or $a(xy)=0.$ Therefore $AL^2=0$. Thus $L^2 \subseteq C_L(A)$
\end{proof}

\vspace{5mm}

\begin{lem}\label{lem5}
$L^2$ is nilpotent if $L$ is supersolvable.
\end{lem}

\begin{proof}
Induct on the dimension of $L$. Let $A$ be a minimal ideal of $L$. Then $L^2 \subseteq C_L(A)$ by  Lemma 4. Now let $A \subseteq L^2$. By induction, $(L/A)^2$ is nilpotent because $(L/A)$ is supersolvable. But $(L/A)^2=(L^2+A)/A$. So there exists a $k$ such that $((L^2+A)/A)^k=0$. Therefore $(L^2+A)^k \subseteq A$ for some $k$. Thus $(L^2)^{k+1} \subseteq (L^2+A)^{k+1}=(L^2+A)(L^2+A)^k \subseteq (L^2+A)A=0$. Thus $L^2$ is nilpotent.
\end{proof}

\vspace{5mm}

\begin{thm}\label{thm4}
$L$ is supersolvable if and only if $L^2$ is nilpotent and $L^*_0(a)=L$ for all $a \in L$.
\end{thm}

\begin{proof}
Suppose $L$ is supersolvable. Then $L^2$ is nilpotent by Lemma 5. Let $v_1, \dots , v_n$ be a basis for $L$ such that $\langle v_1, \dots , v_s \rangle$ is an ideal in $L$ for each $s$. Then the characteristic polynomial for $L_a$ is the product of linear factors and the result holds.
\\
Suppose the conditions hold for $L$. Use induction on the dimension of $L$. It is sufficient to show that each minimal ideal has dimension one. Let $A$ be a minimal ideal of $L$. Then $AL^2+L^2A \trianglelefteq M$ and is in both $L^2$ and $A$. Either $AL^2+L^2A=0$ or $AL^2+L^2A=A$. If $AL^2+L^2A=A$, then A=$AL^2+L^2A=(AL^2+L^2A)L^2+L^2(AL^2+L^2A)$. Repeating the process shows that the term never becomes 0. This contradicts the fact that $L^2$ is nilpotent. Hence $AL^2+L^2A=0$ and $AL^2=0=L^2A$. Since $0=(xy)a=x(ya)-y(xa)$, it follows that $L_xL_y=L_yL_x$ on $A$. Since $A$ is a minimal ideal of $L$, either $R_x=-L_x$ for all $x \in L$ or $R_x=0$ for all $ x \in L$ when $R_x$ acts on $A$. Hence all $R_x , L_y$ commute on $A$. That is $R_xL_y=-L_xL_y=-L_yL_x=L_yR_x$ or $R_xL_y=0=L_yR_x$ and $R_xR_y=L_xL_y=L_yL_x=R_yR_x$ on $A$. Since $L^*_0(a)=L, L_a$ can be triangularized on $A$ as can $R_a$ since $R_a=-L_a$ or $R_a=0$. Since all $L_a, R_a$ can be triangularized on $A$ and they commute, they can be simultaneously triangularized on $A$. Hence there is a common eigenvalue $z$ where $A=\langle z \rangle$. Thus dim$(A)=1$ and induction on $L/A$ yields $L$ supersolvable.
\end{proof}

\vspace{5mm}

\begin{cor}\label{cor2}
Supersolvability is 2-recognizeable in the class of solvable Leibniz algebras.
\end{cor}

\begin{proof}
Suppose that all two generated subalgebras are supersolvable. Then they are strongly solvable and thus so is $L$. Let $a \in L$. For any $b$ in $L$, let $H$ be the subalgebra generated by $a$ and $b$. Then $a$ satisfies the condition in the last theorem in $H$. This extends to $L$, $L^{*}_0(a)=L$. Hence $L$ is supersolvable.
\end{proof}

\vspace{5mm}

As in the last section, we get the usual extension.

\vspace{5mm}

\begin{thm}\label{thm5}
For fields of characteristic 0 or algebraically closed of characteristic greater than 5, supersolvability is 2-recognizeable.
\end{thm}

\vspace{5mm}

We mention a result  that is in a different direction. It is the Leibniz algebra version of a result that holds in Lie and Malcev algebras as well as group theory and is useful in these structures.

\vspace{5mm}

\begin{thm}\label{thm6}
Let $\Phi(L)$ be the Frattini ideal of $L$ and $B$ be an ideal contained in $\Phi(L)$. If $L/B$ is supersolvable, then $L$ is supersolvable. 
\end{thm}

\begin{proof}
Let $x \in L$. Since $L/B=\bar{L}$ is supersolvable, $\bar{L}^2$ is nilpotent by Lemma \ref{lem5}. Hence $L^2/(L^2 \cap B)$ is nilpotent and $L^2$ is nilpotent. Since $B$ is an ideal in $\Phi(L)$, $L^*_1(x) \subseteq B$ for all $ x \in B$. Since $\bar{L}$ is supersolvable, $L^*_1(x) \subseteq B \subseteq \Phi(L)$. Hence $L^*_0(x) + \Phi(L)=L$  for all $x \in L$.  $L^*_0(x)$ is a subalgebra of $L$ by Lemma \ref{lem3}, which forces $L^*_0(x)=L$. Therefore $L$ is supersolvable by Theorem \ref{thm4}.
\end{proof}

\vspace{5mm}

V. A 3-RECOGNIZEABLE CLASS

\vspace{5mm}

The class of abelian by nilpotent Lie algebras is not 2-recognizeable as an example is provided in [12]. The example is the split extension of a Heisenberg Lie algebra, $L$, spanned by $x,y$ and $z$ with $[x,y]=z$, and a one dimensional space spanned by a semi-simple derivation $D$ of $L$ with $D(x)=x$, $D(y)=y$ and $D(z)=2z$. 

\vspace{5mm}
This class of Lie algebras is 3-recognizeable. The same result holds in Leibniz algebras
as in Lie algebras, define $d_k(x,y,z)= (L_z^k(x))(L_z^k(y))$. 

\vspace{5mm}

Let $L$ be a Leibniz algebra such that $d_k(x,y,z)=0$ for all $x,y,z \in L$ and almost all $k$. Then $(L_1(z))(L_1(z))=0$ and $L_1(z)$ is an abelian ideal in $L$.

\vspace{5mm}

\begin{thm}\label{thm7}
Over an infinite field, the following are equivalent:
\begin{enumerate}
\item
$L$ is abelian by nilpotent
\item
$d_k(x,y,z)=0$ for all $x,y,z \in L$ and all $k > \mbox{dim }L$.
\end{enumerate}
\end{thm}

\begin{proof}
(1) clearly implies (2). Assume (2). Let $L$ be a Leibniz algebra such that $d_k(x,y,z)=0$ for all $x,y,z \in L$ and almost all $k$. 
Let $H$ be a Cartan subalgebra of $L$. Then $H$ is the Fitting null component, $L_0(z)$ of $L_z$ for some $z \in H$ [ 3, Theorem 6.5] and $H$ is nilpotent. Furthermore, the condition gives that $(L_1(z))(L_1(z))=0$ and $L_1(z)$ is an abelian ideal in $L$. Hence the result holds.
\end{proof}

\vspace{5mm}

\begin{thm}\label{thm8}
The class of abelian by nilpotent Leibniz algebras over an infinite field is 3-recognizeable.
\end{thm}

\begin{proof}
Suppose that all three generated subalgebras of $L$ are abelian by nilpotent and let $H$ be generated by $x,y$ and $z$. Condition (2) of Theorem \ref{thm7} holds for this $x,y$ and $z$. Therefore condition (2) holds in general in $L$ and $L$ is abelian by nilpotent.
\end{proof}

\vspace{5mm}

\vspace{5mm}

REFERENCES

\vspace{5mm}

1. Sh.A. Ayupov,  B.A.Omirov: On Leibniz algebras. Algebra and Operator Theory . Proceedings of the Colloquium in Tashkent, Kluwer (1998) 1-12

2. D.W. Barnes, M.L.Newell: Some theorems on saturated formations of soluble Lie algebras. Math. Z. 115 (1970) 179-187

3. D.W.Barnes: Some theorems on Leibniz algebras. Comm. Algebra 39, (2011) 2463-2472

4. D.W. Barnes: Schunk classes of soluble Leibniz algebras, Comm. Algebra 41, (2013) 4046-4065.

5. C. Batten Ray, L. Bosko-Dunbar, A. Hedges, J.T. Hird, K. Stagg, E.Stitzinger: A Frattini theory for Leibniz algebras, Comm. Alg. 41, (2013) 1547-1557

6. L. Bosko, A.  Hedges, J.T. Hird, N. Schwartz,, K. Stagg: Jacobson's refinement of Engel's theorem for Leibniz algebras. Involve 4, (2011) 293-296.

7. K. Bowman, D. Towers, V. Varea: Two generator subalgebras of Lie algebras, Linear and Multilinear Algebra 55, (2007) 429-438

8. R. Brandl:  Zur Theorie der Untergruppenabgeschlossenen Formationen: Endliche Variet�ten, J. Algebra 73 (1981) 1-22

9. R. Brandl:  On finite abelian by nilpotent groups. J. Algebra 86 (1984) 439-44410. 

10. A. Elduque:  On supersolvable Malcev algebras, Comm. in Algebra 14 (1986), 311-321

11. J. Loday:  Une version non commutative des algebres de Lie: les algebres do Leibniz. Enseign Math. 39, 269-293 (1993)

12. K. Moneyhun, E. Stitzinger: Some finite varieties of Lie algebras, J. Algebra 143, (1991) 173-178.

13. E. Stitzinger: Supersolvable Malcev algebras, J. Algebra 103 (1986) 69-79.

\end{document}